\documentclass[12pt, twoside
]{article}
\usepackage[cp1251]{inputenc}
\newcommand{\CopyName}{B.\ N.\ Khabibullin\footnote{This  work was financially supported by the Russian Science Foundation (Project No. 18-11-00002).}} 
\newcommand{\NAME}{B.\ N.\ KHABIBULLIN} %
\newcommand{\Year}{} 
\newcommand{\rightheadtext}{INTEGRALS OF SUBHARMONIC FUNCTIONS AND THEIR DIFFERENCES} 
\usepackage[english,russian,ukrainian]{babel}
\usepackage{ms_we4a}

\renewcommand{\refname}{\refnam}
\usepackage{graphicx}
\usepackage{enumerate}
\usepackage{xcolor}
\usepackage{upgreek}

\usepackage[pdftex,unicode,colorlinks,linkcolor=blue,citecolor=red,bookmarksopen,
pdfhighlight=/N]{hyperref}

\usepackage{hyperref}

\newtheorem*{thN}{Rolf~Nevanlinna Theorem}
\newtheorem*{lemGS}{Grishin\,--\,Sodin Lemma on Small Intervals}

\newtheorem*{ThSI}{Theorem on small intervals with weight}
\newtheorem*{LemmaA}{Lemma A}
\newtheorem*{maintheorem}{Main Theorem}
\newtheorem*{mainlemma}{Main Lemma}
\renewcommand{\leq}{\leqslant} 
\renewcommand{\geq}{\geqslant}
\newcommand{\rad}{\text{\tiny\rm rad}}
\newcommand{\RR}{\mathbb{R}} 
\newcommand{\CC}{\mathbb{C}} 
\newcommand{\NN}{\mathbb{N}} 
\newcommand{\ZZ}{\mathbb{Z}}

\DeclareMathOperator{\mes}{mes}
\DeclareMathOperator{\ess}{ess} 
\DeclareMathOperator{\sbh}{sbh} 
\DeclareMathOperator{\Meas}{Meas} 
\DeclareMathOperator{\dd}{\,{\mathrm d\!}}

\vs
\newcommand{\tit}{INTEGRALS OF SUBHARMONIC FUNCTIONS AND THEIR DIFFERENCES WITH WEIGHT OVER SMALL SETS ON A RAY} 
\date{}
\begin{document}
\hbox to \textwidth{\footnotesize
}
\vspace{0.3in}
\textup{\scriptsize{УДК 517.574  + 517.547.2}} \vs 
\markboth{{\NAME}}{{\rightheadtext}}
\begin{center} \textsc {\CopyName} \end{center}
\begin{center} \renewcommand{\baselinestretch}{1.3}\bf {\tit} \end{center}

\vspace{20pt plus 0.5pt} {\abstract{ \noindent B.\ N.\ Khabibullin.\ 
\textit{Integrals of subharmonic functions and their differences with weight over small sets on a ray}\matref 
\vspace{3pt} 

Let  $E$ be a measurable subset  in a segment $[0,r]$ in the positive part of the real axis in the complex plane, and 
$U=u-v$  be the difference  of subharmonic functions $u\not\equiv -\infty$ and $v\not\equiv -\infty$ on the complex plane. An integral of the maximum on circles centered at zero of $U^+:=\sup\{0,U\} $ or $|u|$ over $E$ with a function-multiplier $g\in L^p(E) $ in the integrand is estimated, respectively, in terms of the characteristic function $T_U$ of $U$ or the maximum of  $u$ on circles centered at zero, and also in terms of the linear Lebesgue measure of  $E$ and  the $ L^p$-norm of $g$. 
Our main theorem develops the proof of one of the classical theorems of Rolf Nevanlinna in the case $E=[0,R]$, given in the classical monograph by Anatolii A. Gol'dberg and Iosif V. Ostrovskii, and also generalizes analogs of the Edrei\,--\,Fuchs Lemma on small arcs for small intervals from the works of A.\,F.~Grishin, M.\,L.~Sodin, T.\,I.~Malyutina.
Our estimates are uniform in the sense that the constants in these estimates do not depend on $U$ or $u$, provided that $U$ has an integral normalization near zero or $u(0)\geq 0$, respectively.
}} \vsk

\renewcommand{\refname}{\refnam}

\vskip10pt

\section{Introduction}

\subsection{Origins and research subject}\label{Ss1_1} 
One of the classical theorems of Rolf Nevanlinna can be considered as the original source of our results in this article \cite[pp. 24--27]{RNevanlinna}.  
Anatolii Asirovich Gol'dberg and Iosif Vladimiriovich Ostrovskii indicate the last reference in their classic monograph
 \cite[Notes, Ch. 1]{GOe}. 
The original source \cite{RNevanlinna} remained inaccessible to us. But the mentioned theorem is stated with a complete proof in the monograph of A.\,A.~Gol'dberg and I.\,V.~Ostrovskii \cite[Ch. 1, Theorem 7.2]{GOe}. We give this  result exactly in their formulation
 and notations.
 
Let $f$ be a meromorphic function on the {\it complex plane\/} $\CC$ with the {\it real axis\/} $\RR$, 
\begin{subequations}\label{TN}
	\begin{align}
	M(r,f)&:=\max\Bigl\{\bigl|f(z)\bigr| \colon |z|=r\Bigr\}, \quad r\in \RR^+:=\{x\in \RR \colon x\geq 0\}, 
	\tag{\ref{TN}M}\label{{TN}M}\\
	\intertext{and with the {\it Nevanlinna characteristics}}
	T(r, f)&\underset{\text{\tiny $r\in\RR^+$}}{:=}m(r,f)+N(r,f),  
	\tag{\ref{TN}T}\label{{TN}T}
	\\
	m(r,f)&\underset{\text{\tiny$r\in\RR^+$}}{:=}\frac{1}{2\pi}\int_0^{2\pi} \ln^+\bigl|f(re^{i\varphi})\bigr| \dd \varphi, 
	\quad \ln^+x\underset{\text{\tiny $x\in\RR^+$}}{:=}\max \{\ln x,0\}, 
	\tag{\ref{TN}m}\label{{TN}m}\\
	N(r,f)&\underset{\text{\tiny $r\in\RR^+$}}{:=}\int_{0}^{r}\frac{n(t,f)-n(0,f)}{t}\dd t+n(0,f)\ln r, 
	\tag{\ref{TN}N}\label{{TN}N}
	\end{align}
\end{subequations} 
where   $n(r,f)$ is the number of poles of $f$ in the closed disc $\overline D(r):=\{z\in \CC\colon |z|\leq r\}$, taking into account the multiplicity.
\begin{thN}[{\rm  \cite[Ch. 1, Theorem 7.2]{GOe}}] 
	Let  $f(z)$  be a meromorphic function, $k>1$ be a real number. Then 
	\begin{equation}\label{eN}
	\frac{1}{r}\int_0^r\ln^+ M(t,f)\dd t\leq C(k)T(kr,f), 
	\end{equation}
where the constant $C(k)>1$ depend on  $k$ only.   
\end{thN}
But \cite[Ch.~1, Proof of Theorem 7.2]{GOe} uses \cite[гл. 1, лемма 7.1]{GOe} for $R':=\sqrt{k}r $, which was proved only for  $R'>R>1$. It turns out to be essentially. So, for  meromorphic function
\begin{subequations}\label{Tn}
	\begin{align}
	f(z)&\underset{\text{\tiny $z\in \CC$}}{\equiv}\frac{1}{z}, \qquad M(r,f)\underset{\text{\tiny $r>0 $}}{\overset{\eqref{{TN}M}}{\equiv}}\frac{1}{r}, 
	\tag{\ref{Tn}M}\label{{Tn}f}
	\\
	m(r,f)&\underset{\text{\tiny $r>0$}}{\overset{\eqref{{TN}m}}{\equiv}}\ln^+ \frac{1}{r}, \quad
	N(r,f) \underset{\text{\tiny $r>0$}}{\overset{\eqref{{TN}N}}{\equiv}}\ln r,\quad  
	T(r,f) \underset{\text{\tiny $r>0$}}{\overset{\eqref{{TN}T}}{\equiv}}\ln^+r,  
	\tag{\ref{Tn}T}\label{{Tn}T}
	\\
	\intertext{for  the left-hand side of \eqref{eN}, we have}
	\frac{1}{r}\int_0^r&\ln^+ M(t,f)\dd t\underset{\text{\tiny $r>0$}}{\overset{\eqref{{Tn}f}}{\equiv}}
	\frac{1}{r}\int_0^r\ln^+ \frac{1}{t}\dd t
	\underset{\text{\tiny $r>0$}}{\equiv}
	1+\ln^+ \frac{1}{r}\underset{\text{\tiny $r>0$}}{\geq}  1.
	\tag{\ref{Tn}I}\label{{Tn}M}
	\end{align}
\end{subequations}
Thus, if \eqref{eN} is satisfied, then $C(k)\ln^+kr \geq 1$. But this is impossible if $0\leq r\leq 1/k$.
For the constant $C(k)$ independent of $r$, this is possible only if the requirement of the form $r\geq r_0> 0 $ is added to the inequality \eqref{eN}, and the constant $C(k)$ also depends from the choice of a fixed number 
$r_0>0$. For the sake of fairness, we note that the known cases of application of the Rolf Nevanlinna Theorem are considered, as a rule, only for the cases $ r \to + \infty $ or $r\geq 1$.

Integrals over small subsets on arc or ray intervals are also widely used in the theory of entire and meromorphic functions.  
The starting point of these second-type estimates is  the A.~Edrei and  W.\,H.\,J.~Fuchs Lemma on Small Arcs \cite[{\bf 2}, Lemma III, {\bf 9}]{EF}, which has found important applications in the theory of meromorphic functions, reflected, in particular, in   \cite[Ch.~1, Theorems 7.3, 7.4]{GOe}. 
A variation on the Edrei\,--\,Fuchs Lemma on Small Arcs is a Lemma on Small Intervals by A.\,F.~Grishin and 
M.\,L.~Sodin. The authors note  that its proof repeats verbatim the proof of the Edrei\,--\,Fuchs Lemma  on Small Arcs, and therefore it is presented in \cite[Lemma 3.1]{GrS} without proof, but with interesting applications
\cite[Lemma 3.2, Theorem 3.1]{GrS}. 
We formulate the Grishin\,--\,Sodin Lemma on Small Intervals also in the literal translation of the author's version.

Denote by  $\mes E$  the {\it linear Lebesgue  measure\/} of  $E\subset \RR$, and 
$E(R):=E\cap [1,R)$.  
\begin{lemGS}[{\rm \cite[Lemma 3.1]{GrS}}] Let  $f$ be a meromorphic function on $\CC$, $E\subset [1,+\infty)$. Then
	\begin{equation}\label{eGS}
	\frac{1}{r}\int_{E(r)}\ln^+M(t,f)\dd t\leq 
	C\frac{k}{k-1}\Bigl(\frac{\mes E(r)}{r}\ln
	\frac{2r}{\mes E(r)} \Bigr)T(kr,f),
	\end{equation}
where $C$ is an absolute constant.  
\end{lemGS}
A version of the Grishin\,--\,Sodin Lemma on Small Intervals for subharmonic functions, but only \textit{of finite order,\/} is proved in the joint work of A.\,F.~Grishin and T.\,I.~Malyutina
\cite{GrM}. The Grishin\,--\,Malyutina Lemma on Small Intervals  found several important applications in proofs of key results of \cite[theorems 2, 4]{GrM} and \cite[4.3, Lemma 4.2]{KhaShmAbd20}. We do not give this version, because it is embedded in our joint with L.\,A. Gabdrakhmanova main result from \cite[Theorem 1 (on small intervals)]{GabKha20}, where  is discussed in detail \cite[Conclusion of the Grishin\,--\,Malyutina  Theorem]{GabKha20}. 

In this article, we consider integrals over subsets of intervals on a ray, but  without estimates of integrals over small arcs on circles. In particular, we generalize the Rolf Nevanlinna Theorem, as well as the Grishin\,--\,Sodin Lemma on Small Intervals. 
Our Main Theorem on small intervals with  integral $L^p$-norms for functions-multipliers is formulated in SubSec. \ref{Ss1_2}.
Generalizations of the previous results are given in several directions.  First, we consider differences of subharmonic functions, i.e.,  $\delta$-subharmonic functions of arbitrary growth on the plane. 
We give a simple correction to the summands dictated by counterexample \eqref{Tn}. Second, in the integrand, a multiplier function of the class $L^p $ is allowed for $1<p \leq \infty $, and the integrals are estimated using the $L^p $-norm with corresponding changes in the contribution of small subsets  $E\subset \RR^+$. Third, the estimates of integrals over intervals are in a certain sense uniform with respect to the class of all ($\delta$-)subharmonic functions with an integral normalization near zero. This will make it possible to translate them in the future into estimates of integrals of plurisubharmonic functions and their differences in a multidimensional complex space both for intervals on rays with common beginnings and for subsets on these intervals.

\subsection{Basic definitions and a recent  theorem on small intervals}\label{Ss1_2} 

This subsection can be referred to as needed. Our designations may differ from those used above.
We write single-point sets without curly braces, if this does not cause confusion.
So, $\NN:=\{1,2,\dots\}$ is the set of  {\it natural\/}  numbers and    
$\NN_0:=0\cup \NN$, $\ZZ:=(-\NN)\cup \NN_0$   is the set of {\it integers\/},
and  $\overline \RR$ is the {\it extended real axis\/} with 
$-\infty:=\inf \RR$ and $+\infty:=\sup \RR$, 
$-\infty \leq x\leq +\infty$ for each $x\in \overline \RR$, 
$ -(\pm\infty)=\mp\infty$, $ \overline \RR^+:=\RR^+\cup +\infty$, and 
\begin{equation}\label{actR}
\begin{split}
x+&(+\infty)=+\infty \text{ for $x\in \overline \RR\!\setminus\!-\infty$},
\; x+(-\infty)=-\infty\text{ for }x\in \overline \RR\!\setminus\!+\infty,
\\
x\cdot &(\pm\infty):=\pm\infty=:(-x)\cdot (\mp\infty) \text{ for $x\in \overline \RR^+\!\setminus\!0$}, 
\\
\frac{\pm x}{0}&:=\pm\infty\text{ for $x\in  \overline \RR^+\!\setminus\!0$}, \; 
\frac{x}{\pm\infty}:=0\text{ for $x\in  \RR$},  \quad\text{but } 0\cdot \pm\infty:=0
\end{split}
\end{equation}   
unless otherwise stated; $x^+:=\max\{0,x\}=:(-x)^-$ for $x\in \overline \RR$; 
$\inf \varnothing:=+\infty$ and $\sup  \varnothing:=-\infty$ for the {\it empty set} $\varnothing$.  An {\it interval\/} $I\subset \overline \RR$ is a   {\it connected set\/}
with {\it ends\/} $\inf I\in \overline \RR$ and $\sup I\in \overline \RR$;
$(a,b):=\{x\in {\overline \RR} \colon a<x<b\}$, $[a,b]:=\{x\in {\overline \RR} \colon a\leq x\leq b\}$,
$[a,b):=[a,b]\!\setminus\! b$, $(a,b]:=[a,b]\!\setminus\! a$.
$ D(z,r):=\{z' \in \CC \colon |z'-z|<r\}$ is an {\it open disc,\/} $\overline  D(z,r):=\{z' \in \CC \colon |z'-z|\leq r\}$ 
is a {\it closed disc,} 
$\partial \overline D(z,r):=\{z' \in \CC \colon |z'-z|=r\}$ is the  
{\it circle with center\/ $z\in \CC$ of radius\/ $r\in \overline \RR^+$;}  $D(z,0)=\varnothing$, $\overline  D(z,0)=\partial \overline D(z,0)=z$,  $D(z,+\infty)=\CC$. Besides, $D(r):=D(0,r), \overline D(r):= \overline D(0,r)$, 
$\partial \overline D(r):=\partial \overline D(0,r)$. 

Given a function  $f\colon X\to {\overline \RR}$,    $f^+:=\sup\{0,f\}$ and  $f^-:=(-f)^+$ are  \textit{positive\/} and \textit{negative  parts\/} of function $f$, respectively; $|f|:=f^++f^-$.  

Given  $S\subset \CC$, $\sbh(S)$ is the class of  all {\it subharmonic\/}  on an open  neighbourhood of $S$.  The class  $\sbh(S)$ contains  the \textit{trivial  $(-\infty)$-function\/} $\boldsymbol{-\infty}$. We set   
$\sbh_*(S):=\sbh(S)\!\setminus\! \boldsymbol{-\infty}$ for connected subset $S\subset \CC$.

By $\lambda $ we denote the {\it linear Lebesgue measure\/} on $\RR $ and its restrictions on arbitrary {\it $\lambda$ -measurable subsets\/} $S\subset {\overline \RR}$, by setting $\lambda (\pm\infty):=0$. We also use the notation $\mes S: =\lambda (S) $ from subsection \ref {Ss1_1}. The concepts \textit{almost everywhere, measurability\/} and {\it integrability\/} mean $\lambda$-almost everywhere, $\lambda$-measurability and $\lambda $-integrability, respectively.

 We denote by
\begin{equation}\label{essg} 
\ess \sup_S f:=\inf \Bigl\{a\in \RR\colon \lambda\bigl(\{x\in S\colon
f(x)>a\}\bigr)=0 \Bigr\}
\end{equation}
essential upper bound of measurable  function $f$ defined almost everywhere on  $S$, and
\begin{subequations}\label{Lp}
	\begin{align}
	L^{\infty}(S)&:=\bigl\{f \colon \|f\|_{L^{\infty}(S)}:=\ess\sup_S |f|<+\infty \bigr\}, 
	\tag{\ref{Lp}$_{\infty}$}\label{Linty}
	\\
	\intertext{and for numbers $p>1$,}
	L^p(S)&:=\left\{f \colon \|f\|_{L^p(S)}:=\left(\int_S |f|^p\dd \lambda\right)^{1/p}<+\infty \right\}, \quad p\in (1,+\infty), 
	\tag{\ref{Lp}p}\label{Lpp}
	\\
	\intertext{together with the numbers $q$ associated with $p$ by equality}
	\frac1{p}&+\frac1{q}=1, \quad 1<q=\frac{p}{p-1}<+\infty, 
	\quad \text{but $q:=1$ if  $p=\infty$.}
	\tag{\ref{Lp}q}\label{Lq}
	\end{align} 
\end{subequations}

If $S:=I$ is an interval with ends $a\leq b$, then for {\it Lebesgue integral\/} of  $f$ over the interval $I$ we will use two forms of notation
\begin{equation}\label{intIf}
\int_If\dd \lambda:=:\int_a^b f(t) \dd t.
\end{equation} 

For $ r\in \RR^+ $ and an arbitrary function $ v\colon \partial\overline D(0, r) \to {\overline \RR}$, we define
\begin{equation}\label{Cu}
{\sf M}_v(r):=\sup_{|z|=r}v(z),  \quad {\sf C}_v(r)
:=\frac{1}{2\pi}\int_0^{2\pi} v(re^{is})\dd s  
\end{equation}
The latter is the {\it average over\/} the {\it circle\/} $\partial\overline D(0, r)$ for $v$,
 if the  function $s\mapsto re^{is}$ is integrable on $[0,2\pi]$; ${\sf C}_v(0):={\sf M}_v(0)={\sf C}_v(0)=v(0)$. 
For the properties of the characteristics ${\sf M}_v$ and ${\sf C}_v $ in the case of a subharmonic function $v$, see \cite[2.6]{Rans}, \cite[2.7]{HK}.
So, for meromorphic functions $f$ on $\CC$, we have
\begin{equation*}
M(r,f)\overset{\eqref{{TN}M}}{=}{\sf M}_{\ln |f|}(r), 
\quad m(r,f)\overset{\eqref{{TN}m}}{=}{\sf C}_{\ln^+|f|}(r).
\end{equation*}

\begin{ThSI}[{\rm \cite[Theorem  1, Remark 1.1]{GabKha20}}]
There is an absolute constant  $a\geq 1$ such that  
	\begin{enumerate}[{\rm (i)}]
		\item[{\rm [{\bf u}]}]\label{li} for any subharmonic function  $u\in \sbh_*(\CC)$, i.\,e., $v\not\equiv -\infty$ on $\CC$, 
		\item[{\rm [{\bf r}]}]\label{lii} for any numbers   $0\leq r_0\leq  r<R<+\infty$,
		\item[{\rm [{\bf E}]}]\label{liii} for any measurable subset  $E\subset [r,R]$,
		\item[{\rm [{\bf g}]}]\label{liv} for any measurable function   $g$
on  $E$, 
		\item[{\rm [{\bf b}]}]\label{lb} for any number    $b\in (0,1]$,
	\end{enumerate} 
the following inequality is fulfilled
	\begin{equation}\label{iend}
	\begin{split}
	\int_{E}{\sf M}_{|u|} g\dd \lambda&\leq \left(\frac{a}{b}\ln \frac{a}{b}\right)
	\Bigl( {\sf M}_u\bigl((1+b)R\bigr)	+2{\sf C}_u^-(r_0)\Bigr)
	\|g\|_{L^{\infty}(E)}
	\\
	&\times \underset{m_{\infty}(E; R,b)}{\underbrace{\left(\mes E+\min\{\mes E, 3bR\} \ln \frac{3ebR}{\min\{\mes E, 3bR\}}\right)}}
	\end{split}
	\end{equation}
where $m_{\infty}(E; R,b)\leq 2\mes E$ when  $\mes E> 3bR$, and 
\begin{equation*}
m_{\infty}(E; R,b) \leq 2 \mes E \ln \frac{3ebR}{\mes E},
\quad\text{if $\mes E\leq 3bR$.}
\end{equation*}
\end{ThSI}

\section{Main Theorem}\label{Main}

For a Borel subset $S\subset \CC$, the set  of all Borel, or Radon,  positive measures $\mu\geq 0$  on $S$
is denoted by $\Meas^+(S)$, 
and $\Meas(S):=\Meas^+(S)-\Meas^+(S)$ is the set  of all {\it charges,\/} or signed measures, on $S$.
For a measure  $\mu\in \Meas^+\bigl(\overline D(R)\bigr)$, we set
\begin{subequations}\label{murad}
\begin{flalign}
\mu^{\rad}(r)&:=\mu\bigl(\overline D(r)\bigr)\in \RR^+, 
\quad 0\leq r\leq R,
\tag{\ref{murad}m}\label{{murad}m}
\\
{\sf N}_{\mu}(r,R)&:=\int_{r}^{R}\frac{\mu^{\rad}(t)}{t}\dd t\in \overline \RR^+, 
\quad 0\leq r\leq R,
\tag{\ref{murad}N}\label{{murad}N}
\end{flalign}
\end{subequations} 

For $ 0\leq r\leq R\in \RR^+ $ and an arbitrary function $ v\colon \partial\overline D(r)\cup \partial\overline D(R)\to {\overline \RR}$, we define
\begin{equation}\label{MCr}
{\sf C}_v(r,R)\overset{\eqref{Cu}}{:=}{\sf C}_v(R)-{\sf C}_v(r)=\frac{1}{2\pi}\int_0^{2\pi} \bigl(v(Re^{is})-v(re^{is})\bigr)\dd s 
\end{equation}
provided that ${\sf C}_v(R)$ and ${\sf C}_v(r)$ are well defined.

If $D\subset \CC$ is a domain and  $u\in \sbh_*(D)$, then there is  its {\it Riesz measure\/} 
\begin{equation}\label{df:cm}
\varDelta_u:= \frac{1}{2\pi} {\bigtriangleup}  u\in \Meas^+( D), 
\end{equation}
where ${\bigtriangleup}$  is  the {\it Laplace operator\/}  acting in the sense of the  theory of distribution or generalized functions. This definition of the Riesz measures carries over naturally to $u\in \sbh_*(S)$ for connected subsets $S\subset \CC$.
If $v\in \sbh_*\bigl(\overline D(R)\bigr)$, then, by the Poisson\,--\,Jensen\,--\,Privalov formula,  we have
\begin{equation}\label{CN}
{\sf C}_v(r,R)={\sf N}_{\varDelta_v}(r,R)
\quad\text{for all $0<r<R<+\infty$}.
\end{equation}

Let $U=u-v$ be a difference of subharmonic functions $u,v \in \sbh_*\bigl(\overline D(0, R)\bigr)$, i.\,e., a  {\it $\delta$-subharmonic non-trivial ($\not\equiv\pm\infty$) function\/} on $\overline D(R)$ with the {\it Riesz charge\/} $\varDelta_{U}=\varDelta_u-\varDelta_v$ \cite{Arsove53}, \cite{Arsove53p}, \cite{Gr}, \cite[3.1]{KhaRoz18}. A representation $U=u_U-v_U$ with $u_U,v_U\in \sbh_*\bigl(\overline D(0, R)\bigr)$ is {\it canonical\/} if the  Riesz measure $\varDelta_{u_U}$ of $u_U$ is  the {\it upper variation\/}  $\varDelta_U^+$ of $\varDelta_U$  and  the Riesz measure  $\varDelta_{v_U}$ of $v_U$ is  the {\it lower  variation\/}  $\varDelta_U^-$ of $\varDelta_U$. 
The canonical representation for $U$ is defined up to the harmonic function added simultaneously to each of the representing subharmonic functions $u_U$ and $v_U$.
We define  a {\it characteristic function\/} of this $\delta$-subharmonic function $U$ as 
a function of two variables 
\begin{equation}\label{T}
\begin{split}
{\sf T}_U(r,R)&:={\sf C}_{\sup\{u_U,v_U\}}(r,R)=
{\sf C}_{U^+}(r,R)+{\sf C}_{v_U}(r,R)\\
&\overset{\eqref{CN}}{=}{\sf C}_{U^+}(r,R)+{\sf N}_{\varDelta_U^-}(r,R), \quad 0<r\leq R\in \RR^+. 
\end{split}
\end{equation}
This characteristic function $T_U$ is already uniquely defined for all $0<r\leq R<+\infty$  by positive values in $\RR^+$, and is also  increasing and convex with respect to $\ln$ in the second variable $R$, but is decreasing  in the first  variable $r\leq R$.

\begin{maintheorem}\label{th1} 
Let  $0< r_0< r<+\infty$, $1<k\in \RR^+$,  $E\subset [0,r]$ be measurable, 
$g\in L^p(E)$, where $1<p\leq \infty$ and $q\in [1,+\infty)$ is from \eqref{Lq}, 
$U\not\equiv \pm\infty$  be a $\delta$-subharmonic  non-trivial functions on $\CC$, 
and $u\not\equiv -\infty$ be a subharmonic function on $\CC$.  Then 
\begin{subequations}\label{1}
\begin{flalign}
\frac{1}{r}\int_{E} {\sf M}_{U}^+(t)g(t)\dd t\leq
4q\frac{k}{k-1} \bigl({\sf T}_{U}(r_0,kr)+{\sf C}_{U^+}(r_0)\bigr) \|g\|_{L^p(E)}\frac{\sqrt[q]{\mes E}}{r}
\ln\frac{4kr}{\mes E},
\tag{\ref{1}T}\label{inDl+}
\\ 
\frac{1}{r}\int_{E} {\sf M}_{|u|}(t)g(t)\dd t\leq
5q\frac{k}{k-1} \bigl({\sf M}_{u^+}(kr)+{\sf C}_{u^-}(r_0)\bigr) \|g\|_{L^p(E)}\frac{\sqrt[q]{\mes E}}{r}
\ln\frac{4kr}{\mes E}.
\tag{\ref{1}M}\label{uM}
\end{flalign}
\end{subequations}
\end{maintheorem}

\section{Lemmata and Proof of Main Theorem}

\begin{lemma}\label{lem2} Let $\mu\in \Meas^+\bigl(\overline D(R)\bigr)$ . Then
\begin{equation}\label{N}
\mu^{\rad}(r)\leq  \frac{R}{R-r}{\sf N}_{\mu}(r,R)\quad \text{for each $0\leq r\leq R$.}
\end{equation} 
\end{lemma}
\begin{proof}[Proof] By definitions \eqref{murad}, we have
\begin{equation*}
\mu^{\rad}(r)= \int_{r}^{R}\frac{\mu^{\rad}(r)}{t} \dd t\biggm/\int_r^R \frac{1}{t}\dd t 
 \leq {\sf N}_{\mu}(r,R)\biggm/\int_r^R \frac{1}{R}\dd t
= \frac{R}{R-r}{\sf N}_{\mu}(r,R),
\end{equation*}
and we obtain \eqref{N}.
\end{proof}

\begin{lemma}\label{lem1} Let\/  $0\leq r<R<+\infty$,  $E\subset [0,r]$ be measurable, $U=u-v$  be a difference of subharmonic  functions $u, v\in \sbh_* \bigl(\overline D(R)\bigr)$,    
$\varDelta_v$ be the Riesz measure of $v$,  and $g\in L^p(E)$. Then 
\begin{equation}\label{inDR}
\int_{E} {\sf M}_{U}^+(t)g(t)\dd t\leq \biggl(\frac{R+r}{R-r}
{\sf C}_{U^+}(R)\sqrt[q]{\mes E}+\varDelta_v^{\rad}(R)
\sup_{0\leq x\leq R} \Bigl\|\ln\frac{2R}{\bigl|\cdot-x\bigr|}\Bigr\|_{L^q(E)}\biggr) \|g\|_{L^p(E)}. 
\end{equation}
\end{lemma}
\begin{proof}[Proof] For $w\in E\subset \overline D(r)$, bу the Poisson\,--\,Jensen formula \cite[4.5]{Rans}, we have
\begin{multline*}
U(te^{ia})=\frac{1}{2\pi}\int_0^{2\pi}U(Re^{is})\Re \frac{Re^{is}+te^{ia}}{Re^{is}-te^{ia}}\dd s
-\int_{D(R)}\ln \Bigl|\frac{R^2-zte^{-ia}}{R(te^{ia}-z)}\Bigr|\dd \varDelta_u(z) \\
+\int_{D(R)}\ln \Bigl|\frac{R^2-zte^{-ia}}{R(te^{ia}-z)}\Bigr|\dd \varDelta_v(z)  
\leq \frac{R+r}{R-r}{\sf C}_{U^+}(R) +\int_{D(R)}\ln \frac{2R}{\bigl|t-|z|\bigr|}\dd \varDelta_v(z) 
\end{multline*}
where the right-hand side of the inequality is positive and independent of $a\in[0,2\pi)$. Hence, by integrating, we get 
\begin{equation*}
\int_{E} {\sf M}_{U}^+(t)g(t)\dd t\leq \frac{R+r}{R-r}
{\sf C}_{U^+}(R)\int_E |g|(t)\dd t+
\int_{D(R)}\int_E\ln\frac{2R}{\bigl|t-|z|\bigr|} |g|(t)\dd t\dd \varDelta_v(z).
\end{equation*}
Therefore, by H\"older's inequality, we obtain
\begin{multline*}
\int_{E} {\sf M}_{U}^+(t)g(t)\dd t\leq \frac{R+r}{R-r}
{\sf C}_{U^+}(R)\|g\|_{L^p(E)}(\mes E)^{1/q}\\
+\varDelta_v\bigl(D(R)\bigr)
\|g\|_{L^p(E)}\sup_{z\in D(R)} \Bigl\|\ln\frac{2R}{\bigl|\cdot-|z|\bigr|}\Bigr\|_{L^q(E)}.
\end{multline*}
The latter gives \eqref{inDR}.  
\end{proof}

\begin{lemma}\label{lemln}
Let $q\in \RR^+$,  $0<A\in \RR^+$, and $a\in (0,A/e]$. Then
\begin{equation}\label{intq}
\int_0^a\ln^q\frac{A}{x} \dd x\leq (1+q^{q+1})
a\ln^q\frac{A}{a}. 
\end{equation}
\end{lemma}
\begin{proof}[Proof] 
We denote by  $\lfloor q\rfloor:=\max\{ n\in \ZZ\colon n\leq q\}$  the {\it integer part\/} of $q$.
Evidently, 
\begin{equation}\label{lnq}
\ln\frac{A}{x}\geq 1 \quad \text{if $x\in (0,A/e]$, } 
\end{equation}
We integrate the integral from \eqref{intq} by parts $\lfloor q \rfloor + 1 $ times: 
\begin{multline*}
\int_0^a\ln^q\frac{A}{x} \dd x=a\ln^q\frac{A}{a}+
qa\ln^{q-1}\frac{A}{a}+q(q-1)a\ln^{q-2}\frac{A}{a}+\dots\\
\dots+q(q-1)\cdots (q-\lfloor q\rfloor +1)\ln^{q-\lfloor q\rfloor}\frac{A}{a}
+q(q-1)\cdots (q-\lfloor q\rfloor)\int_0^a \ln^{q-\lfloor q\rfloor-1}\frac{A}{x}\dd x\\
\leq \Bigl(a\ln^q\frac{A}{a}\Bigr)
\left(1+ \frac{q}{\ln\frac{A}{a}}+\frac{q(q-1)}{\ln^{2}\frac{A}{a}}+\dots\right.\\
\left.\dots+\frac{q(q-1)\cdots (q-\lfloor q\rfloor +1)}{\ln^{\lfloor q\rfloor}\frac{A}{a}}
+\frac{q(q-1)\cdots (q-\lfloor q\rfloor)}{a\ln^q\frac{A}{a}}\int_0^a \ln^{q-\lfloor q\rfloor-1}\frac{A}{x}\dd x\right)\\
\overset{\eqref{lnq}}{\leq}
\underset{P_q}{\underbrace{\bigl(1+ q+q(q-1)+\dots+q(q-1)\cdots (q-\lfloor q\rfloor +1)
+q(q-1)\cdots (q-\lfloor q\rfloor)\bigr)}} \;a\ln^q\frac{A}{a}.
\end{multline*}
We have  a recurrent formula $P_q=1+qP_{q-1}$ for $1\leq q\in \RR^+$. 
Therefore, 
\begin{equation*}
P_q=\sum_{j=0}^{\lfloor q\rfloor}q^j+q^{\lfloor q\rfloor}\bigl(q-\lfloor q\rfloor\bigr) \leq
1+q^{\lfloor q\rfloor+1}\leq 1+q^{q+1} \quad\text{for each $p\in \RR^+$}.
\end{equation*}
Thus, we obtain \eqref{intq}. 
\end{proof}

\begin{lemma}\label{lem4} For $E\subset [0,r]\subset [0,R]$, $q\geq 1$, and $0\leq x\leq R$, we have
 \begin{equation}\label{sqq}
\Bigl\|\ln\frac{2R}{\bigl|\cdot-x\bigr|}\Bigr\|_{L^q(E)}\leq 
2q\sqrt[q]{\mes E}\ln \frac{4R}{\mes E},
\end{equation}
\end{lemma}
\begin{proof}[Proof]  We use 
\begin{LemmaA}[{\rm \cite[Lemma 7.2]{GOe}}] Let 
$f\colon (-a,a)\to {\overline \RR}$ be an even integrable function  on $(-a,a)$ decreasing on $(0,a)$, $E\subset (-a,a)$  be a measurable subset. Then 
\begin{equation}\label{intI}
\int_{E}f\dd \lambda\leq 
2\int_0^{\lambda(E)/2} f(t)\dd t.
\end{equation}
\end{LemmaA}

By Lemma A  we obtain
\begin{equation*}
\int_{E}\ln^q\frac{2R}{\bigl|t-x\bigr|}\dd t=\int_{E-x}\ln^q\frac{2R}{|t|}\dd t
\leq 2\int_0^{\lambda(E-x)/2} \ln^q\frac{2R}{t}\dd t= 2\int_0^{\lambda(E)/2} \ln^q\frac{2R}{t}\dd t,
\end{equation*}
where $a:=\lambda(E)/2\leq r/2\leq R/2\leq 2R/e=:A/e$. Hence, by  Lemma \ref{lemln}, we have 
\begin{equation*}
\int_{E}\ln^q\frac{2R}{\bigl|t-x\bigr|}\dd t\leq  
2(1+q^{q+1})\frac{\lambda(E)}{2}\ln^q\frac{2R}{\lambda(E)/2}=(1+q^{q+1})(\mes E)\ln^q \frac{4R}{\mes E}
\end{equation*}
and 
$(1+q^{q+1})^{1/q}\leq (2q^{q+1})^{1/q}=q\bigl((2q)^{1/2q}\bigr)^2\leq 2q$ for $q\geq 1$,  
since the function $x\underset{\text{\tiny $x\in \RR^+$}}{\longmapsto} x^x$ is decreasing in $[1/e,+\infty)$.  Thus, 
\begin{equation*}
\Bigl\|\int_{E}\ln^q\frac{2R}{\bigl|\cdot-x\bigr|}\Bigr\|_{L^q(E)}\leq  
\sqrt[q]{1+q^{q+1}}\sqrt[q]{\mes E}\ln \frac{4R}{\mes E}\leq 
2 q\sqrt[q]{\mes E}\ln \Bigl(\frac{4R}{\mes E}\Bigr)
\end{equation*}
 which gives \eqref{sqq}.
\end{proof}

\begin{mainlemma} Let  $0< r<+\infty$, $0<b\in \RR^+$,  $E\subset [0,r]$ be measurable, $U=u-v$  be a difference of subharmonic  functions $u, v\in \sbh_* \Bigl(\overline D\bigl((1+b)^2r\bigr)\Bigr)$, 
 and $g\in L^p(E)$, where $1<p\leq +\infty$ and $q\in [1,+\infty)$ is from \eqref{Lq}. Then 
\begin{multline}\label{inDl}
\int_{E} {\sf M}_{U}^+(t)g(t)\dd t\leq
q\frac{2+b}{b}
\Bigl(
{\sf C}_{U^+}\bigl((1+b)r\bigr)
+{\sf N}_{\varDelta_v}\bigl((1+b)r, (1+b)^2r\bigr)\Bigr)\\
\times  \|g\|_{L^p(E)}\sqrt[q]{\mes E}
\ln\frac{4(1+b)r}{\mes E}. 
\end{multline}
\end{mainlemma}

\begin{proof}[Proof] We set $R:=(1+b)r$. By Lemma \ref{lem1}, Lemma \ref{lem2} with $(1+b)^2r$ instead of $R$ and $R$ instead of $r$, and Lemma \ref{lem4}, we have
\begin{multline*}
\int_{E} {\sf M}_{U}^+(t)g(t)\dd t\leq \Bigl(\frac{2+b}{b}
{\sf C}_{U^+}\bigl((1+b)r\bigr)\sqrt[q]{\mes E}
\\
+\frac{1+b}{b}{\sf N}_{\varDelta_v}\bigl((1+b)r,(1+b)^2r\bigr) 
2q\sqrt[q]{\mes E}\ln \frac{4(1+b)r}{\mes E} \Bigr) \|g\|_{L^p(E)}\\
\leq \frac{2+b}{b}
\Bigl({\sf C}_{U^+}\bigl((1+b)r\bigr)
+{\sf N}_{\varDelta_v}\bigl((1+b)r,(1+b)^2r\bigr) \Bigr) \|g\|_{L^p(E)}
q\sqrt[q]{\mes E}\ln \frac{4(1+b)r}{\mes E}, 
\end{multline*}
which gives \eqref{inDl}.
\end{proof}
\begin{proof}[Proof of Main Theorem] 
We can assume that $U=u-v$ is the canonical representation of $U$. 
Consider a number $b>0$ such that $(1+b)^2=k$. By Main Lemma, we have
\begin{multline*}
\int_{E} {\sf M}_{U}^+(t)g(t)\dd t\leq q\frac{\sqrt{k}+1}{\sqrt{k}-1}
\Bigl( {\sf C}_{U^+}(\sqrt{k}r) +{\sf N}_{\varDelta_v}(\sqrt{k}r, kr)\Bigr)
\|g\|_{L^p(E)}\sqrt[q]{\mes E} \ln\frac{4\sqrt{k}r}{\mes E}\\
\leq q\frac{\sqrt{k}+1}{\sqrt{k}-1}
\Bigl( \underset{{\sf T}_{U}(r_0, \sqrt{k}r)}{\underbrace{{\sf C}_{U^+}(r_0, \sqrt{k}r)
+{\sf N}_{\varDelta_v}( r_0, \sqrt{k}r)}}+{\sf C}_{U^+}(r_0)\Bigr.
\\
\Bigl.
+\underset{{\sf T}_{U}(r_0, kr)}{\underbrace{{\sf C}_{U^+}(r_0, kr)+{\sf N}_{\varDelta_v}(r_0, kr)}}\Bigr)
 \|g\|_{L^p(E)}\sqrt[q]{\mes E}
\ln\frac{4\sqrt{k}r}{\mes E}
\\ 
\leq q\frac{2k}{k-1}
\Bigl( {\sf T}_{U}(r_0, \sqrt{k}r)+{\sf T}_{U}(r_0, kr)+{\sf C}_{U^+}(r_0)\Bigr)
\|g\|_{L^p(E)}\sqrt[q]{\mes E}
\ln\frac{4kr}{\mes E},
\\ 
\leq q\frac{4k}{k-1}
\Bigl( {\sf T}_{U}(r_0, kr)+{\sf C}_{U^+}(r_0)\Bigr)
\|g\|_{L^p(E)}\sqrt[q]{\mes E}
\ln\frac{4kr}{\mes E},
\end{multline*}
and, by definition \eqref{T}, obtain  \eqref{inDl+}. 

Evidently, for any function $u$ with values in $\overline \RR$, we have 
\begin{equation}\label{M+}
{\sf M}^+_{u}={\sf M}_{u^+}, \quad  {\sf M}_{|u|}\leq {\sf M}_{u^+}+{\sf M}_{u^-}={\sf M}_{u^+}+{\sf M}_{(-u)^+},
\end{equation}
If $u\in \sbh_*(\CC)$, then, under conditions of Main Theorem, the function  ${\sf M}_u^+$ is increasing, and, by H\"older's inequality, 
\begin{equation}\label{M+E}
\int_{E} {\sf M}_{u^+}(t)g(t)\dd t\leq
{\sf M}_{u^+}(r) \|g\|_{L^p(E)}\sqrt[q]{\mes E}.
\end{equation}
For  $U_u:=0-u$, the difference  $0-u$ is the canonical representation of $\delta$-subharmonic non-trivial function $U_u$ and  we have
\begin{equation}\label{TCM}
{\sf T}_{U_u}(r,R)\overset{\eqref{T}}{=}  {\sf C}_{\sup\{0,u\}}(r,R)={\sf C}_{u^+}(r,R)\leq  {\sf C}_{u^+}(R)
\leq {\sf M}_{u^+}(R).
\end{equation}
Hence, by Main Theorem in part \eqref{inDl+} for $U_u$ in the role of $U$, we obtain  
   
\begin{multline*}
\frac{1}{r}\int_{E} {\sf M}_{(-u)^+}(t)g(t)\dd t\overset{\eqref{M+}}{=}
\frac{1}{r}\int_{E} {\sf M}_{U_u}^+(t)g(t)\dd t\\
\overset{\eqref{inDl+}}{\leq}
\frac{4kq}{k-1} \bigl({\sf T}_{U_u}(r_0,kr)+{\sf C}_{U_u^+}(r_0)\bigr) \|g\|_{L^p(E)}\frac{\sqrt[q]{\mes E}}{r}
\ln\frac{4kr}{\mes E}\\
\overset{\eqref{TCM}}{\leq}
\frac{4kq}{k-1} \bigl({\sf M}_{u^+}(kr)+{\sf C}_{(-u)^+}(r_0)\bigr) \|g\|_{L^p(E)}\frac{\sqrt[q]{\mes E}}{r}
\ln\frac{4kr}{\mes E}. 
\end{multline*}
The latter together with \eqref{M+E} gives \eqref{uM}. 
\end{proof}


\vsk

Bashkir State University

Bashkortostan

 Russian Federation

 khabib-bulat@mail.ru 

\vs

\end{document}